\documentclass[a4paper]{article}
\usepackage[utf8]{inputenc}
\usepackage{indentfirst}
\usepackage{graphicx}
\usepackage{color}
\usepackage{float}
\usepackage{titlesec}
\usepackage{amsmath,amsfonts,amssymb,amsthm,amscd}
\usepackage{mathtools,mathrsfs}
\usepackage{commath}
\usepackage[font=footnotesize,labelfont=it,justification=centering,margin={2cm}]{caption}
\usepackage{enumitem}
\usepackage{hyperref}

\titleformat{\section}{\normalsize\bfseries\centering}{\thesection.}{0.3em}{}
\titleformat{\subsection}[runin]{\normalsize\bfseries}{\thesubsection.}{0.3em}{}

\newtheorem{thm}{Theorem}[section]

\newtheorem{prop}[thm]{Proposition}

\theoremstyle{definition}
\newtheorem{defn}{Definition}[section]
\newtheorem{ex}{Example}[section]
\theoremstyle{remark}
\newtheorem*{rem}{Remark}

\graphicspath{{graphics/}}

\newcommand{\co}{\colon\thinspace}

\title{\textbf{{\large HIGHER-ORDER CONGRUENCE RELATIONS ON AFFINE MOMENT GRAPHS: THE SUBGENERIC CASE}}}
\author{{\normalsize KSENIJA KITANOV}}
\date{}

\begin{document}

\maketitle

\begin{abstract}
We study the structure algebra $\mathcal{Z}$ of the stable moment graph for the case of the affine root system $A_{1}$. The structure algebra $\mathcal{Z}$ is an algebra 
over a symmetric algebra and in particular, it is a module over a symmetric algebra. We study this module structure on $\mathcal{Z}$ and we construct a basis. By 
``setting $c$ equal to zero'' in $\mathcal{Z}$, we obtain the module $\mathcal{Z}_{c=0}$. This module can be described in terms of the finite root system $A_{1}$ and we 
show that it is determined by a set of certain divisibility relations. These relations can be regarded as a generalization of ordinary moment graph relations that define 
sections of sheaves on moment graphs, and because of this we call them higher-order congruence relations.
\end{abstract}


\section{INTRODUCTION} 

In 1979 Kazhdan and Lusztig introduced a special basis of the Hecke algebra attached to a Coxeter system (cf.~\cite{KL79}). Entries of the transition matrix between the 
standard basis and this special basis are given by the Kazhdan-Lusztig polynomials. In the same article a character formula for simple highest weight modules over a 
complex semisimple Lie algebra was conjectured. This character formula was expressed in terms of the Kazhdan-Lusztig polynomials associated with the Weyl group of a 
semisimple Lie algebra. The Kazhdan-Lusztig polynomials together with the Kazhdan-Lusztig conjecture, which was proved independently by Beilinson and Bernstein 
(cf.~\cite{BB81}) and by Brylinski and Kashiwara (cf.~\cite{BK81}), constitute the foundations of geometric representation theory.

From a Coxeter group, a standard parabolic subgroup and a sign Deodhar constructed parabolic (spherical or anti-spherical\textemdash{}depending on the sign) modules over 
the Hecke algebra (cf.~\cite{Deo87}). More precisely, following the notation of~\cite{Soe97}, let $(\mathcal{W}, \mathcal{S})$ be a Coxeter system and $\mathcal{H}$ its 
Hecke algebra over the ring $\mathcal{L}=\mathbb{Z}[v,v^{-1}]$ of Laurent polynomials in one variable with integer coefficients. Let $I \subset \mathcal{S}$ be a subset 
of the set of simple reflections, $\mathcal{W}_{I}=\langle I \rangle$ the parabolic subgroup corresponding to $I$ and $\mathcal{H}_{I}$ the Hecke algebra of 
$(\mathcal{W}_{I}, I)$. For $u \in \{ -v, v^{-1} \}$, $\mathcal{L}=\mathcal{L}(u)$ can be endowed with a structure of an $\mathcal{H}_{I}$-bimodule. After induction two 
right $\mathcal{H}$-modules are obtained: the spherical module $\mathcal{M}=\mathcal{L}(v^{-1}) \otimes_{\mathcal{H}_{I}} \mathcal{H}$ and the anti-spherical module 
$\mathcal{N}=\mathcal{L}(-v) \otimes_{\mathcal{H}_{I}} \mathcal{H}$ (the terms \emph{spherical} and \emph{anti-spherical} are taken from~\cite{LW}). Parabolic (spherical 
or anti-spherical) Kazhdan-Lusztig polynomials appear as elements of the transition matrix between the standard basis and the Kazhdan-Lusztig basis of the parabolic 
(spherical or anti-spherical) Hecke module and they can be regarded as a generalization of the ordinary Kazhdan-Lusztig polynomials (when $I=\emptyset$, i.e. the parabolic 
subgroup $\mathcal{W}_{I}$ is trivial, the parabolic Kazhdan-Lusztig polynomials coincide with the ordinary Kazhdan-Lusztig polynomials).

The Kazhdan-Lusztig polynomials possess fascinating properties, which are not apparent from their definition. Let us mention non-negativity of their coefficients, which 
was proved in full generality by Elias and Williamson (cf.~\cite{EW14}). The same property was established for the parabolic Kazhdan-Lusztig polynomials in~\cite{LW}.

Another distinguished property, proved by Lusztig in~\cite{Lus80}, is the stabilization property of the Kazhdan-Lusztig polynomials. It can be stated in terms of the 
parabolic Kazhdan-Lusztig polynomials corresponding to the spherical module $\mathcal{M}$ ($m$-polynomials in Soergel's notation, cf.~\cite{Soe97}, Theorem 6.1): far 
enough inside the fundamental (Weyl) chamber the $m$-poly\-no\-mi\-als stabilize. In this stabilization phenomenon the generic Kazhdan-Lusztig polynomials ($q$-polynomials 
in Soergel's notation) make their appearance.

Via sheaves on moment graphs (these notions were introduced in~\cite{GKM98} and~\cite{BM01}) the stabilization property of the Kazhdan-Lusztig polynomials is lifted to 
the categorical level. In~\cite{Lan15} the moment graph analogue of this stabilization property\textemdash{}\emph{the stable moment graph}\textemdash{}is defined. It is an
oriented graph whose vertices are given by the alcoves in the fundamental (Weyl) chamber and whose edges are labeled with affine coroots. The stable moment graph 
arises as a subgraph of the parabolic Bruhat graph and with the help of the stable moment graph the stabilization property is established in the categorical framework of 
sheaves on moment graphs: the stalks of indecomposable Braden-MacPherson sheaves on finite intervals of the parabolic Bruhat graph (far enough inside the fundamental (Weyl) 
chamber) stabilize as well (cf.~\cite{Lan15}). In addition, in characteristic $0$, graded ranks of the stalks of indecomposable Braden-MacPherson sheaves (for the definition 
of the graded rank and connection to the Kazhdan-Lusztig polynomials, the reader may refer to~\cite{Fie10}) on the stable moment graph are given by the generic 
Kazhdan-Lusztig polynomials\textemdash{}the previously mentioned $q$-polynomials that, besides appearing in the stabilization phenomenon of the parabolic spherical 
polynomials, also appear as coefficients (up to a factor $\pm1$ and an involution) in the expansion with respect to the standard basis of the periodic Hecke module (for 
more details on the periodic Hecke module, the reader may consult~\cite{Lus80} and~\cite{Soe97}) of certain self-dual elements of a completion of the periodic Hecke module 
(cf.~\cite{Kat85}, \cite{Soe97}, Theorem 6.4.2).

Also via sheaves on moment graphs in~\cite{Fie11} Fiebig constructs a weak categorification of the affine Hecke algebra. The category in this weak categorification is a 
category of special modules over the affine structure algebra that can be interpreted as a category of sheaves on moment graphs and even as Soergel's category of bimodules 
associated with a reflection faithful representation of a Coxeter system (the latter interpretation is established in~\cite{Fie08b} and Soergel's construction can be found 
in~\cite{Soe07}). Furthermore, Fiebig's combinatorial weak categorification of the affine Hecke algebra is generalized to the case of the (parabolic) spherical Hecke module 
in~\cite{Lan14}.

Let us now observe another categorification of the Hecke algebra that is representation-theoretic in its nature\textemdash{}category $\mathcal{O}$ of Bernstein, Gelfand and 
Gelfand. In~\cite{Fie08a} a link between (equivariant) representation theory and theory of sheaves on moment graphs is formed: if a moment graph is associated with a 
non-critical block $\mathcal{O}_{\Lambda}$ of the equivariant category $\mathcal{O}$ over a symmetrizable Kac-Moody algebra, then a certain subcategory of the category 
of sheaves on this associated moment graph is equivalent (as an exact category) to the subcategory of $\mathcal{O}_{\Lambda}$ of modules that admit a finite (equivariant) 
Verma flag. Therefore, the stable moment graph should conjecturally yield information on stabilization phenomena for non-critical singular blocks of the (equivariant) 
category $\mathcal{O}$ for affine Kac-Moody algebras.

Let us add a few words on the structure algebra of the moment graph associated with a root system. Let $k$ be a field and $S$ the symmetric algebra over $k$ associated with 
the coroot lattice. The structure algebra is a commutative algebra over $S$. If $k=\mathbb{C}$, the structure algebra of the moment graph associated with a non-critical 
block of the equivariant category $\mathcal{O}$ over a symmetrizable Kac-Moody algebra is isomorphic to the categorical center of that non-critical block 
(cf.~\cite{Fie08a}). If $\text{char}(k)=0$, the structure algebra of the moment graph associated with a complex equivariantly formal variety with a torus action is 
isomorphic to the torus equivariant cohomology with coefficients in $k$ of that equivariantly formal variety (cf.~\cite{GKM98}).

In this article we study the structure algebra $\mathcal{Z}$ of the stable moment graph for the case of the affine root system $A_{1}$ (the subgeneric case). Let us 
mention that the obtained results apply to an arbitrary affine root system\textemdash{}the studied phenomena occur if we localize at a root direction. In the forthcoming 
article, first we consider the case of the affine root system $A_{2}$. We look at a certain (translated) subgraph of the stable moment graph whose set of vertices is 
obtained by translations of the finite Weyl group orbit of the fundamental alcove by a negative integer multiple of a positive (co)root and we study its structure algebra. 
Here we apply our results from the subgeneric case by ``stretching out'' in a root direction and by ``gluing''. Using the same approach, we then consider the general case. 
In particular, the structure algebra $\mathcal{Z}$ is a module over the polynomial algebra $S$ and we examine this $S$-module structure of $\mathcal{Z}$. We construct an 
$S$-basis (of ``linear-algebraic type'') of $\mathcal{Z}$. Then we regard the same notions for the case of the finite root system $A_{1}$: the associated moment graph 
$\mathcal{G}^{\text{fin}}$, the corresponding symmetric algebra $S^{\text{fin}}$ and the structure algebra $\mathcal{Z}^{\text{fin}}$. After ``setting $c$ equal to zero'' 
($c$ is a central element of the affine Kac-Moody algebra $\widehat{\mathfrak{sl}_{2}}$) in the structure algebra $\mathcal{Z}$, we obtain the $S^{\text{fin}}$-module 
$\mathcal{Z}_{c=0}$. We observe that this $S^{\text{fin}}$-module is contained in $\prod_{j \in \mathbb{Z}_{\geq 0}} \mathcal{Z}^{\text{fin}}$ (we could say that it is 
``locally finite'') and our main result (Theorem~\ref{main}) is that the module $\mathcal{Z}_{c=0}$ is determined by (quite surprising) divisibility relations, which we 
call \emph{higher-order congruence relations}.

\subsection{Contents.} In Section 2 we collect the basic structural results on affine Kac-Moody algebras (e.g. affine roots, the affine Weyl group, alcoves) with the 
purpose of setting the notation. 

In Section 3 we consider affine Bruhat graphs (Definition~\ref{abg}) with special emphasis placed on the parabolic Bruhat graph and its subgraph that contains only 
stable edges\textemdash{}edges whose labels are invariant under translation by an element of the finite coroot lattice, i.e. the stable moment graph (Definition~\ref{smg}). 
We end Section 3 with the definition of the structure algebra $\mathcal{Z}$ of a moment graph (Definition~\eqref{eq:stralg}).

In Section 4 first we construct a basis of the $S$-module $\mathcal{Z}$ of the stable moment graph in the subgeneric case (Propositions~\ref{unvn} and~\ref{bas}). Then 
we define the $S^{\text{fin}}$-module $\mathcal{Z}_{c=0}$ by ``setting $c$ equal to zero'' in the $S$-module $\mathcal{Z}$ (Subsection~\ref{subsec:ctozero}) and we 
establish the main result of this article, Theorem~\ref{main}:

\[
\mathcal{Z}_{c=0} = \Bigg\{ (\mathbf{a}_{j}) \in \prod_{j \in \mathbb{Z}_{\geq 0}} \mathcal{Z}^{\textnormal{fin}}
                         \biggm\vert \ \sum_{j=0}^{m} (-1)^{j} \binom{m}{j} \mathbf{a}_{j} \in ((-\alpha)^{m}, \alpha^{m}) 
                         \mathcal{Z}^{\textnormal{fin}} \ \forall m \in \mathbb{Z}_{\geq 0} \Bigg\}.
\]

\section{PRELIMINARIES}

\subsection{Affine Kac-Moody algebras.} Let $\mathfrak{g}$ be a finite-dimensional simple complex Lie algebra and let $\mathfrak{\widehat{g}}$ be the corresponding 
affine Kac-Moody algebra, which as a vector space is equal to $\mathfrak{g}\otimes_{\mathbb{C}}\mathbb{C}[t,t^{-1}] \oplus \mathbb{C}c \oplus \mathbb{C}d$, where 
$c$ is a central element and $d$ is a derivation operator. 

Let $\mathfrak{h} \subset \mathfrak{g}$ be a Cartan subalgebra. Then the corresponding affine Cartan subalgebra of $\mathfrak{\widehat{g}}$ and its dual are given by 
$\mathfrak{\widehat{h}} = \mathfrak{h} \oplus \mathbb{C}c \oplus \mathbb{C}d$ and $\mathfrak{\widehat{h}^{*}} = \mathfrak{h^{*}} \oplus \mathbb{C}\delta \oplus 
\mathbb{C}\Lambda_{0}$, respectively. The Killing form $(\cdot, \cdot)$ on $\mathfrak{g}$ can be extended to a symmetric non-degenerate bilinear form on 
$\mathfrak{\widehat{g}}$. This form induces a non-degenerate bilinear form on the Cartan subalgebra $\mathfrak{\widehat{h}}$ and consequently establishes an isomorphism 
$\mathfrak{\widehat{h}} \overset{\sim}{\to} \mathfrak{\widehat{h}^{*}}$, which identifies $c$ with $\delta$. We also denote by $(\cdot, \cdot)$ the induced symmetric 
non-degenerate bilinear form on the dual $\mathfrak{\widehat{h}^{*}}$. 

\subsection{Affine roots.} Let $\Delta \subset \mathfrak{h^{*}}$ be the set of roots of $\mathfrak{g}$ with respect to $\mathfrak{h}$. The set $\widehat{\Delta} 
\subset \mathfrak{\widehat{h}^{*}}$ of roots of $\mathfrak{\widehat{g}}$ with respect to $\mathfrak{\widehat{h}}$ is 
\[ \widehat{\Delta} = \{ \alpha + n\delta \mid \alpha \in \Delta,n \in \mathbb{Z} \} \cup \{ n\delta \mid n \in \mathbb{Z},n \neq 0 \}. \] 

If we denote by $\Delta_{+} \subset \Delta$ the subset of chosen positive (finite) roots, then the set of positive affine roots is
\[ \widehat{\Delta}_{+} = \{ \alpha + n\delta \mid \alpha \in \Delta,n \geq 1 \} \cup \Delta_{+} \cup \{ n\delta \mid n \geq 1 \}, \] 
while the set of positive real roots is  
\[ \widehat{\Delta}^{\text{re}}_{+} = \{ \alpha + n\delta \mid \alpha \in \Delta,n \geq 1 \} \cup \Delta_{+}. \]

Furthermore, if $\Pi$ is the set of simple (finite) roots and $\theta \in \Delta^{+}$ is the highest root, then \[ \widehat{\Pi} = \Pi \cup \{ -\theta + \delta \} \] 
is the set of simple affine roots.

\subsection{The Weyl group and its set of alcoves.} For every real root $\alpha + n\delta$ we can define the reflection
\begin{align*}
s_{\alpha+n\delta}\co \mathfrak{\widehat{h}^{*}} &\to \mathfrak{\widehat{h}^{*}} \\
\lambda &\mapsto \lambda - 2\frac{(\lambda, \alpha + n\delta)}{(\alpha, \alpha)}(\alpha + n\delta).
\end{align*} 
This reflection fixes pointwise the hyperplane $(\cdot, \alpha + n\delta)=0$ and maps $\alpha + n\delta$ to $-\alpha-n\delta$. Using the isomorphism $\mathfrak{\widehat{h}}
\overset{\sim}{\to} \mathfrak{\widehat{h}^{*}}$ we may identify the affine coroot $(\alpha+n\delta)^{\vee}$ with $\frac{2(\alpha+n\delta)}{(\alpha,\alpha)}$ and then the 
reflection $s_{\alpha+n\delta}$ is given by $s_{\alpha+n\delta}(\lambda)=\lambda-\langle\lambda,(\alpha+n\delta)^{\vee}\rangle(\alpha+n\delta)$, where 
$\langle\cdot,\cdot\rangle \co \mathfrak{\widehat{h}^{*}}\times\mathfrak{\widehat{h}} \to \mathbb{C}$ is the natural pairing. We denote by $\mathcal{\widehat{W}} 
\subset GL(\mathfrak{\widehat{h}^{*}})$ the affine Weyl group, that is, the subgroup generated by the reflections $s_{\beta}$ for $\beta \in 
\widehat{\Delta}^{\text{re}}_{+}$ and by $\mathcal{\widehat{S}} = \{ s_{\beta} \mid \beta \in \widehat{\Pi} \}$ the set of simple reflections of $\mathcal{\widehat{W}}$. 
Then $(\mathcal{\widehat{W}},\mathcal{\widehat{S}})$ forms a Coxeter system and by $\leq$ we denote the Bruhat order on $\mathcal{\widehat{W}}$.

Let $\mathfrak{h^{*}_{\mathbb{R}}}$ be the linear span over $\mathbb{R}$ of $\Pi$. We set $\mathfrak{\widehat{h}^{*}}_{\mathbb{R}} = \mathfrak{h^{*}_{\mathbb{R}}} \oplus 
\mathbb{R}\delta \oplus \mathbb{R}\Lambda_{0}$. Now we will consider another realization of $\mathcal{\widehat{W}}$, as a group of affine transformations of 
$\mathfrak{h^{*}_{\mathbb{R}}}$. In order to do that, we identify $\mathfrak{h^{*}_{\mathbb{R}}}$ with the affine space $\mathfrak{\widehat{h}^{*}}_{1} \textrm{ mod } 
\mathbb{R}\delta$, where $\mathfrak{\widehat{h}^{*}}_{1} = \{ \lambda \in \mathfrak{\widehat{h}^{*}}_{\mathbb{R}} \mid \langle \lambda , c \rangle = 1 \}$ 
(cf.~\cite[\S 6.6.]{Kac90}). Then $\mathcal{\widehat{W}}$ acts on $\lambda \in \mathfrak{h^{*}_{\mathbb{R}}}$ by
\begin{equation} \label{eq:left}
s_{\alpha+n\delta}(\lambda) = s_{\alpha}(\lambda)-2\frac{n}{(\alpha,\alpha)}\alpha \textrm{ mod } \mathbb{R}\delta = s_{\alpha}(\lambda)-n\alpha^{\vee} \textrm{ mod } 
\mathbb{R}\delta.
\end{equation} 
We denote by $Q^{\vee}$ the coroot lattice of $\mathfrak{g}$ and by $T_{\mu}$ the translation by $\mu \in Q^{\vee}$, i.e. $T_{\mu}(\lambda) = \lambda + \mu$ for 
$\lambda \in \mathfrak{h^{*}_{\mathbb{R}}}$, so~\eqref{eq:left} can be written as
\[ s_{\alpha+n\delta}(\lambda) = T_{-n\alpha^{\vee}}s_{\alpha}(\lambda) \  \textrm{ for }\lambda \in \mathfrak{h^{*}_{\mathbb{R}}}, \]
i.e. $s_{\alpha+n\delta} = T_{-n\alpha^{\vee}}s_{\alpha}$. The group of translations by an element of the coroot lattice is a normal subgroup of the affine Weyl group 
and it holds that $\mathcal{\widehat{W}}=\mathcal{W} \ltimes Q^{\vee}$.

Let us denote by $H_{\alpha, n} \coloneqq \{ \lambda \in \mathfrak{h^{*}_{\mathbb{R}}} \mid (\lambda, \alpha) = -n \}$ the affine hyperplane that $s_{\alpha + n\delta}$ 
fixes pointwise. The connected components of \[\mathfrak{h^{*}_{\mathbb{R}}} \setminus \bigcup_{\alpha+n\delta \in \widehat{\Delta}^{\text{re}}_{+}} H_{\alpha, n}\] are 
called \emph{alcoves} and we denote by $\mathcal{A}$ the set of all alcoves. The \emph{fundamental (Weyl) chamber} is $\mathcal{C}^{+} \coloneqq \{ \lambda \in 
\mathfrak{h^{*}_{\mathbb{R}}} \mid (\lambda, \alpha) > 0 \ \forall \alpha \in \Pi \}$ (an element $\lambda \in \mathcal{C}^{+}$ is called \emph{dominant weight}) and 
let $\mathcal{A}^{+} = \{ A \in \mathcal{A} \mid A \subset \mathcal{C}^{+} \}$ be the set of all alcoves contained in $\mathcal{C}^{+}$. 

There is a one-to-one correspondence between $\mathcal{\widehat{W}}$ and $\mathcal{A}$. Let $A^{+}$ be the unique alcove in $\mathcal{A}^{+}$ whose closure contains the 
zero vector. $A^{+}$ is called the \emph{fundamental alcove} and we have 
\begin{gather*}
A^{+}=\{ \lambda \in \mathfrak{h^{*}_{\mathbb{R}}} \mid 0 < (\lambda, \alpha) < 1 \ \forall \alpha \in \Delta_{+} \}= \\ =\{ \lambda \in 
\mathfrak{h^{*}_{\mathbb{R}}} \mid 0 < (\lambda, \alpha) \ \forall \alpha \in \Pi , \ (\lambda, \theta) < 1 \}
\end{gather*}
(cf.~\cite[\S 4.3.]{Hum90}). Since the left action (given by~\eqref{eq:left}) of the affine Weyl group $\mathcal{\widehat{W}}$ on $\mathcal{A}$ is simply transitive 
(cf.~\cite[\S 4.5.]{Hum90}), the bijection between $\mathcal{\widehat{W}}$ and $\mathcal{A}$ is given by 
\begin{equation} \label{eq:WA}
w \quad \mapsto \quad wA^{+}. 
\end{equation}

Finally, let us observe that the Bruhat order on $\mathcal{\widehat{W}}$ induces a partial order on $\mathcal{A}$: for $A, B \in \mathcal{A}$ where $A=xA^{+}$, $B=yA^{+}$, 
$x, y \in \mathcal{\widehat{W}}$, $A \leq B \iff x \leq y$. It is again called \emph{Bruhat order}.

\section{MOMENT GRAPHS}

In this section we introduce the notion of a moment graph over a lattice and then we consider particular moment graphs that are important for the representation theory 
of affine Kac-Moody algebras. We conclude this section with the definition of the structure algebra of a moment graph.

\begin{defn} (cf.~\cite{Fie16})
Let $Y \cong \mathbb{Z}^{r}$ be a lattice of finite rank. A \emph{moment graph} over the lattice $Y$ is the datum $\mathcal{G} = 
(\mathcal{V}, \mathcal{E}, \leq, l)$ where:
 \begin{enumerate}
 \item[\textnormal{(i)}] $(\mathcal{V}, \mathcal{E})$ is a directed graph without loops or multiple edges,
 \item[\textnormal{(ii)}] $\leq$ is a partial order on $\mathcal{V}$ such that if $x, y \in \mathcal{V}$ and $E : x \rightarrow y$,
 then $x \leq y$,
 \item[\textnormal{(iii)}] $l \co \mathcal{E} \to Y \setminus \{0\}$ is a map called the labeling. 
 \end{enumerate}
\end{defn}

Using the notation of the previous section, for any subset $I$ of $\mathcal{\widehat{S}}$ we define now the \emph{affine Bruhat graph} $\mathcal{\widehat{G}}^{I}$ in 
the following way.
\begin{defn} (cf.~\cite{Fie11}) \label{abg}
$\mathcal{\widehat{G}}^{I}$ is the moment graph over the affine coroot lattice $\widehat{Q}^{\vee} = Q^{\vee} \oplus \mathbb{Z}$ given by
 \begin{enumerate}
 \item[\textnormal{(i)}] $\mathcal{V} \coloneqq \mathcal{\widehat{W}}^{I}$ where $\mathcal{\widehat{W}}^{I}$ is the set of minimal length representatives of
 the left cosets of $\langle I \rangle$ in $\mathcal{\widehat{W}}$,
 \item[\textnormal{(ii)}] $\leq$ is the Bruhat order on $\mathcal{\widehat{W}}$ and $\mathcal{E} \coloneqq \{ x \rightarrow y \mid x < y, \  \exists \beta 
 \in \widehat{\Delta}^{\text{re}}_{+}, \  \exists w \in \langle I \rangle \textrm{ such that } y = s_{\beta}xw \}$,
 \item[\textnormal{(iii)}] $l(x \rightarrow s_{\beta}xw) \coloneqq \beta^{\vee}$.
 \end{enumerate}
\end{defn}

If $I = \emptyset$, the corresponding affine Bruhat graph is called the \emph{regular Bruhat graph} of $\mathfrak{\widehat{g}}$. If $I = \mathcal{S} \coloneqq 
\{ s_{\beta} \mid \beta \in \Pi \}$ (the set of simple reflections of the finite Weyl group $\mathcal{W}$), we denote by $\mathcal{\widehat{G}}^{\text{par}} \coloneqq 
\mathcal{\widehat{G}}^{\mathcal{S}}$ the \emph{parabolic Bruhat graph} of $\mathfrak{\widehat{g}}$.

Let us observe that we can approach the set of vertices of $\mathcal{\widehat{G}}^{\text{par}}$ in two ways: 
\begin{enumerate}
 \item[\textnormal{(a)}] \label{itm:a} via the identification with the finite coroot lattice $Q^{\vee}$ (obtained by the orbit-stabilizer theorem) or 
 \item[\textnormal{(b)}] via the identification with the set $\mathcal{A}^{+}$ of alcoves contained in the fundamental chamber $\mathcal{C}^{+}$ (obtained intermediately 
 via the set of minimal length representatives of the right cosets of $\mathcal{W}$ in $\mathcal{\widehat{W}}$ and~\eqref{eq:WA}).
\end{enumerate}

\begin{ex} \label{corlat}
Let $\mathfrak{\widehat{g}} = \widehat{\mathfrak{sl}_{2}}$. Then the set of positive real roots is 
\[\widehat{\Delta}^{\text{re}}_{+} = \{ \pm \alpha + n\delta \mid n \geq 1 \} \cup \{ \alpha \}\] where $\alpha$ is the unique positive root of $\mathfrak{sl}_{2}$ 
with the property $(\alpha,\alpha)=2$. Simple affine reflections are $s_{0} \coloneqq s_{-\theta+\delta} = s_{-\alpha+\delta} = T_{\alpha}s_{\alpha}$ and 
$s_{1} \coloneqq s_{\alpha}$, so we have 
\begin{align*}
T_{\alpha}&=s_{0}s_{1},\\
s_{-\alpha+n\delta}&=T_{n\alpha}s_{\alpha}=(s_{0}s_{1})^{n}s_{1}=(s_{0}s_{1})^{n-1}s_{0},\\
s_{\alpha+n\delta}&=T_{-n\alpha}s_{\alpha}=(s_{1}s_{0})^{n}s_{1}
\end{align*}
for $n \in \mathbb{Z}_{>0}$. Using the isomorphism $\mathfrak{\widehat{h}} \overset{\sim}{\to} \mathfrak{\widehat{h}^{*}}$ we obtain 
\[(\pm \alpha + n\delta )^{\vee}=\pm \alpha^{\vee}+\frac{2n}{(\alpha,\alpha)}c=\pm \alpha + nc.\] 
By~\hyperref[itm:a]{(a)} the set of vertices of $\mathcal{\widehat{G}}^{\text{par}}$ is $\mathbb{Z}\alpha^{\vee} = \mathbb{Z}\alpha$ and we have
\begin{equation*}
\begin{split}
s_{-\alpha+(n+n')\delta}(n\alpha)=s_{\alpha-(n+n')\delta}(n\alpha)=T_{(n+n')\alpha}s_{\alpha}(n\alpha)=\\
=-n\alpha+(n+n')\alpha=n'\alpha, \quad n, n' \in \mathbb{Z}.
\end{split}
\end{equation*} 
It follows that $\mathcal{\widehat{G}}^{\text{par}}$ (considered as an undirected graph) is a complete graph and the labeling is given by
\[
l(n\alpha \ \textrm{\textemdash} \ n'\alpha)=\begin{cases}
                                           -\alpha + (n+n')c \quad \textrm{if } n+n' > 0 \\
                                           \alpha - (n+n')c \quad \textrm{if } n+n'\leq 0. 
                                           \end{cases}
\] 
Regarding the direction of the edges, notice that
\[
n\alpha=(s_{0}s_{1})^{n-1}s_{0}(0) \quad \textrm{ for } n > 0\\
\]
and
\[
n\alpha=(s_{1}s_{0})^{\abs{n}}(0) \quad \textrm{ for } n \leq 0,
\]
so we have
\[ 
n\alpha < n'\alpha \quad \iff \quad \abs{n} < \abs{n'} \enspace \veebar \enspace n'=-n < 0.
\]
Hence the interval $[n\alpha, (-n-2)\alpha], \ n > 0$ of the parabolic Bruhat graph $\mathcal{\widehat{G}}^{\text{par}}$ is as in Figure~\ref{fig:parabolic}.
\end{ex}

\begin{figure}[!htb]
\centering
\def\svgwidth{0.75\textwidth}
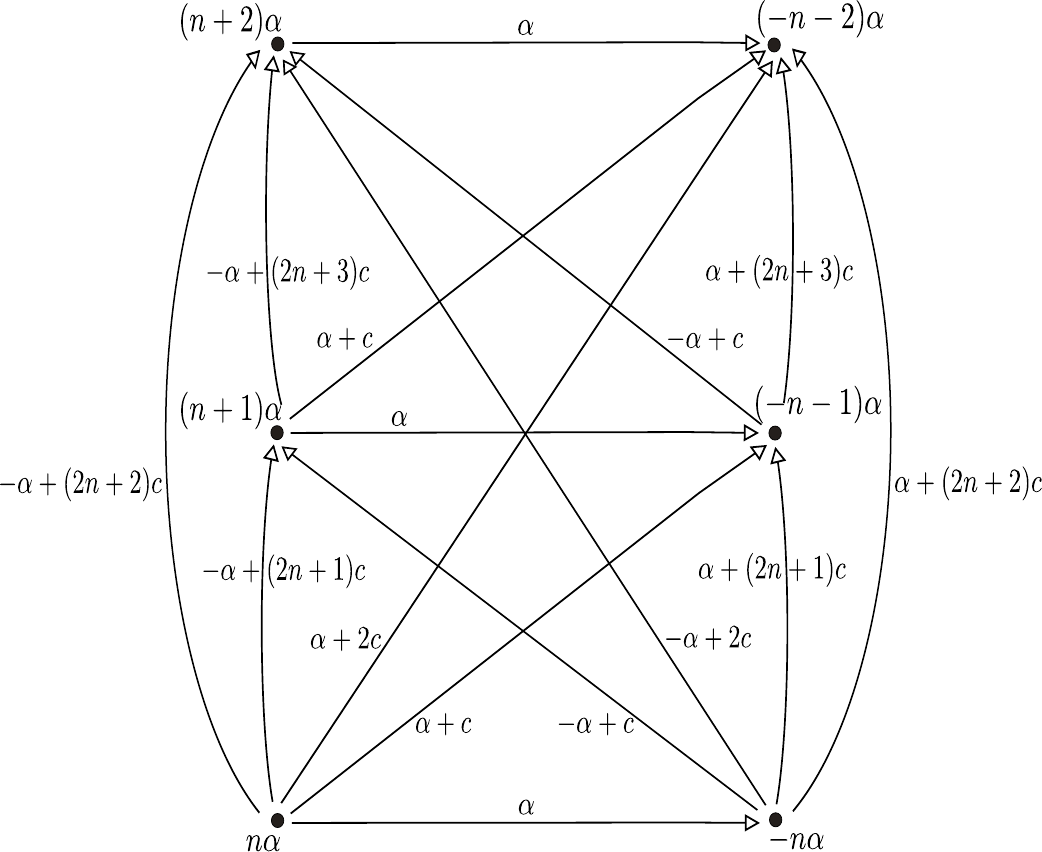
\caption{The interval $[n\alpha, (-n-2)\alpha], n > 0$ of the parabolic Bruhat graph of $\widehat{\mathfrak{sl}_{2}}$} \label{fig:parabolic}
\end{figure}

The set of edges of the parabolic Bruhat graph $\mathcal{\widehat{G}}^{\text{par}}$ far enough in the fundamental chamber consists of \emph{stable} and 
\emph{non-stable} edges. It holds that the labels of stable edges are invariant under translation by $\mu \in Q^{\vee}$, but that does not hold for the labels of 
non-stable edges (Figure~\ref{fig:stable}). This motivates the following definition. 

\begin{figure}[!htb]
\centering
\def\svgwidth{0.5\textwidth}
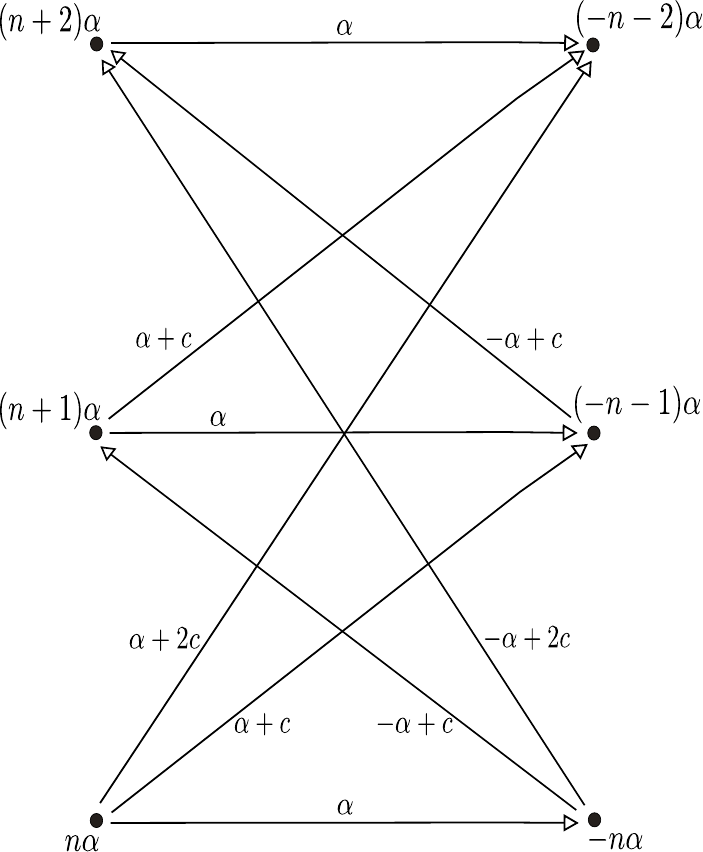
\caption{The subgraph of the parabolic Bruhat graph of $\widehat{\mathfrak{sl}_{2}}$ that contains only stable edges} \label{fig:stable}
\end{figure}

\begin{defn} (cf.~\cite{Lan15}) \label{smg}
The stable moment graph $\mathcal{\widehat{G}}^{\text{stab}}=(\mathcal{V}, \mathcal{E}, \leq, l)$ is the moment graph over the affine coroot lattice $\widehat{Q}^{\vee}$ 
given by
\begin{enumerate}
 \item[\textnormal{(i)}] $\mathcal{V} \coloneqq \mathcal{A}^{+}$,
 \item[\textnormal{(ii)}] $\leq$ is the Bruhat order on $\mathcal{A}$ and $\mathcal{E} \coloneqq \{ xA^{+} \rightarrow yA^{+} \mid 
 x \leq y, \  \exists \beta \in \widehat{\Delta}^{\text{re}}_{+} \ \textrm{ such that } \ y = xs_{\beta} \}$,
 \item[\textnormal{(iii)}] $l(xA^{+} \rightarrow xs_{\beta}A^{+}) \coloneqq \beta^{\vee}$.
\end{enumerate}
\end{defn}

Let $\mathcal{G}$ be a moment graph defined over the lattice $Y$ and $k$ a field. Let us denote by $Y_{k} = Y \otimes_{\mathbb{Z}} k$ the $k$-vector space spanned by 
$Y$ and by $S=S(Y_{k})$ the symmetric algebra of $Y_{k}$, which is a polynomial algebra (over $k$) of rank $\dim_{k}Y_{k}$. The \emph{structure algebra} $\mathcal{Z}$ of 
$\mathcal{G}$ is
\begin{align} \label{eq:stralg}
\mathcal{Z} \coloneqq \bigg\{ (z_{x}) \in \prod_{x \in \mathcal{V}} S \biggm\vert \ & \textrm{ for each edge } E : x \ \textrm{\textemdash} \ y \nonumber \\
& \quad z_{x}-z_{y}=l(E)s \ \textrm{ for some } s \in S \bigg\}.
\end{align}
$\mathcal{Z}$ is indeed an $S$-algebra: the addition and the multiplication are given componentwise and $S$ acts on $\mathcal{Z}$ via the diagonal action.

\section{THE STRUCTURE ALGEBRA OF THE STABLE MOMENT GRAPH IN THE SUBGENERIC CASE}

Let $\mathfrak{\widehat{g}} = \widehat{\mathfrak{sl}_{2}}$. Then the (coroot) lattice $Y$ is $\mathbb{Z} \alpha \oplus \mathbb{Z} c$ (Example~\ref{corlat}) and the 
corresponding symmetric algebra $S$ over the field $k=\mathbb{C}$ is $\mathbb{C}[\alpha, c]$. Let $\mathcal{Z}=\mathcal{Z}(\mathcal{\widehat{G}}^{\text{stab}}_{>0\alpha})$ 
be the structure algebra of the moment graph $\mathcal{\widehat{G}}^{\text{stab}}_{>0\alpha}$ (Figure~\ref{fig:stable}). Let us remark that omitting the vertex 
$0\alpha$ corresponds to being far enough inside the fundamental (Weyl) chamber, away from its walls. 

We explore now the $S$-module structure of the structure algebra $\mathcal{Z}$. The first step in that direction consists of constructing perhaps the most natural 
(from a linear algebra perspective) elements $u_{n},\, n \in \mathbb{Z}_{\geq 0}$ and $v_{n},\, n \in \mathbb{Z}_{>0}$ of the structure algebra $\mathcal{Z}$ 
(Proposition~\ref{unvn}). Afterwards, we will show that these elements form a basis of the $S$-module $\mathcal{Z}$ (Proposition~\ref{bas}). 

We use the following conventions: a product with only one factor evaluates to that factor, while a product with no factors, i.e. an empty product evaluates to $1$.

\begin{prop} \label{unvn}
Let $u_{n},\, n \in \mathbb{Z}_{\geq 0}$ and $v_{n},\, n \in \mathbb{Z}_{>0}$ be defined by
 \[
 (u_{n})_{j\alpha} \coloneqq 
 \begin{dcases}
  0 & \text{if } \ 0 < \abs{j} \leq n\\
  \prod_{l=1}^{n} (-\alpha + lc) & \text{if } \ j=n+1\\
  \prod_{l=1}^{n} (\alpha + lc) & \text{if } \ j=-n-1\\
  k_{j-n-1, n}\cdot\prod_{l=j-n}^{j-1} (-\alpha + lc) & \text{if } \ j \geq n+2\\
  k_{\,\abs{j}-n-1, n}\cdot\prod_{l=\,\abs{j}-n}^{\abs{j}-1} (\alpha + lc) & \text{if } \ j \leq -n-2
 \end{dcases}
 \]
where 
\begin{equation} \label{eq:k_i} 
k_{i, n}=\binom{n+i}{n}, \ i \in \mathbb{Z}_{>0}
\end{equation}
and 
 \[
 (v_{n})_{j\alpha} \coloneqq 
 \begin{dcases}
  0 & \text{if } \ 0 < \abs{j} \leq n-1\\
  0 & \text{if } \ j=n\\
  \prod_{l=0}^{n-1} (\alpha + lc) & \text{if } \ j=-n\\
  (a_{j-n, n}\alpha+(j-n)b_{j-n, n}c)\cdot\prod_{l=j-n+1}^{j-1} (-\alpha + lc) & \text{if } \ j \geq n+1\\
  b_{\,\abs{j}-n, n}\cdot\prod_{l=\,\abs{j}-n}^{\abs{j}-1} (\alpha + lc) & \text{if } \ j \leq -n-1
 \end{dcases}
 \]
where 
\begin{align} \label{eq:ab_i}
a_{i, n} &= \binom{n+i-1}{n}(n-1), \nonumber \\
b_{i, n} &= \binom{n+i-1}{n}\frac{i-(i-1)n}{i}, \ i \in \mathbb{Z}_{>0}. 
\end{align}
Then the elements $u_{n},\, n \in \mathbb{Z}_{\geq 0}$ and $v_{n},\, n \in \mathbb{Z}_{>0}$ belong to the structure algebra $\mathcal{Z}$. 
\end{prop}

\begin{rem}
The restricted and ``flipped'' elements $u_{n},\, n \in \mathbb{Z}_{\geq 0}$ resemble to the sections described in~\cite{Lan15}, Lemma 6.5 and Lemma 6.6, where the 
coefficients are given by falling factorials, while the coefficients of $u_{n}$ are given by binomial coefficients. 
\end{rem}

\begin{rem}
One can think of $u_{n}$ as

\[
u_{n}=\begin{pmatrix}
  \vdots & \vdots \\
  k_{3, n} \cdot \displaystyle\prod_{l=4}^{n+3} (-\alpha + lc) & k_{3, n} \cdot \displaystyle\prod_{l=4}^{n+3} (\alpha + lc) \\
  k_{2, n} \cdot \displaystyle\prod_{l=3}^{n+2} (-\alpha + lc) & k_{2, n} \cdot \displaystyle\prod_{l=3}^{n+2} (\alpha + lc) \\
  k_{1, n} \cdot \displaystyle\prod_{l=2}^{n+1} (-\alpha + lc) & k_{1, n} \cdot \displaystyle\prod_{l=2}^{n+1} (\alpha + lc) \\
  \displaystyle\prod_{l=1}^{n} (-\alpha + lc) & \displaystyle\prod_{l=1}^{n} (\alpha + lc) \\
  0 & 0 \\
  \vdots & \vdots \\
  0 & 0 \\
  0 & 0
 \end{pmatrix}
 \]
where the first $n \ (n \geq 0)$ rows are zero and $v_{n}$ as

\[
v_{n}=\begin{pmatrix}
  \vdots & \vdots \\
  (a_{3, n}\alpha+3b_{3, n}c) \cdot \displaystyle\prod_{l=4}^{n+2} (-\alpha + lc) & b_{3, n} \cdot \displaystyle\prod_{l=3}^{n+2} (\alpha + lc) \\
  (a_{2, n}\alpha+2b_{2, n}c) \cdot \displaystyle\prod_{l=3}^{n+1} (-\alpha + lc) & b_{2, n} \cdot \displaystyle\prod_{l=2}^{n+1} (\alpha + lc) \\
  (a_{1, n}\alpha+b_{1, n}c) \cdot \displaystyle\prod_{l=2}^{n} (-\alpha + lc) & b_{1, n} \cdot \displaystyle\prod_{l=1}^{n} (\alpha + lc) \\
  0 & \displaystyle\prod_{l=0}^{n-1} (\alpha + lc) \\
  0 & 0 \\
  \vdots & \vdots \\
  0 & 0 \\
  0 & 0
 \end{pmatrix}
 \]
where the first $n-1 \ (n \geq 1)$ rows are zero. This very convenient notation will be used throughout the rest of the article.
\end{rem}

\begin{proof}
First we prove that $u_{n},\, n \in \mathbb{Z}_{\geq 0}$ are elements of the structure algebra $\mathcal{Z}$. Let $n, i, i' \in \mathbb{Z}_{\geq 0}$ 
(Figure~\ref{fig:stable_edge}). If $i=i'$, then the corresponding structure algebra condition is satisfied, i.e. 
\[
\alpha \bigm\vert k_{i, n} \cdot \prod_{l=i+1}^{i+n} (-\alpha + lc) - k_{i, n} \cdot \prod_{l=i+1}^{i+n} (\alpha + lc),
\]
where we set $k_{0, n} \coloneqq 1$. If $i > i'$, we need to prove that 
\[
\alpha-(i'-i)c=\alpha+(i-i')c \bigm\vert k_{i, n} \cdot \prod_{l=i+1}^{i+n} (\alpha + lc) - k_{i', n} \cdot \prod_{l=i'+1}^{i'+n} (-\alpha + lc).
\]

\begin{figure}[!htb]
\centering
\def\svgwidth{0.65\textwidth}
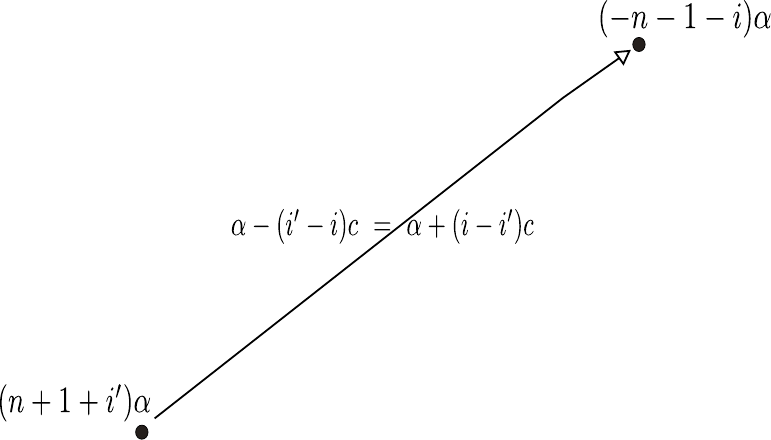
\caption{An edge of the stable moment graph of $\widehat{\mathfrak{sl}_{2}}$} \label{fig:stable_edge}
\end{figure}

It is enough to show that 
\begin{equation} \label{eq:zero}
k_{i, n} \cdot \prod_{l=i+1}^{i+n} ((i'-i)c + lc) - k_{i', n} \cdot \prod_{l=i'+1}^{i'+n} ((i-i')c + lc) = 0.
\end{equation}
We have
\begin{equation*}
\begin{split}
& k_{i, n} \cdot \prod_{l=i+1}^{i+n} ((i'-i)c + lc) = \binom{n+i}{n}(i'+1)(i'+2)\cdots(i'+n)c^{n}=\\
& =\binom{n+i}{n}\binom{n+i'}{n}n!c^{n}.
\end{split}
\end{equation*}
Observe that $\binom{n+i}{n}\binom{n+i'}{n}n!c^{n}$ is symmetric in $i$ and $i'$, so that~\eqref{eq:zero} follows immediately from the previous computation.

Now if $i < i'$, observe that the only thing that changes is a global sign, so that the divisibility relations to be satisfied are the same, and hence already proved. 
Regarding the differences $(u_{n})_{j\alpha}-(u_{n})_{j'\alpha}$ for $\abs{j} \leq n$ or $\abs{j'} \leq n$, let us note that the corresponding divisibility relations 
are obviously satisfied, and with this we have proved that $u_{n} \in \mathcal{Z}$ for every $n \in \mathbb{Z}_{\geq 0}$.

Now we prove that $v_{n},\, n \in \mathbb{Z}_{>0}$ are also elements of the structure algebra $\mathcal{Z}$. Let $n \in \mathbb{Z}_{>0}$. First, 
let $i' \in \mathbb{Z}_{>0}$ (Figure~\ref{fig:stable_edge_}). We need to prove that 
\[
-\alpha+i'c \bigm\vert (a_{i', n}\alpha+i'b_{i', n}c) \cdot \prod_{l=i'+1}^{i'+n-1} (-\alpha + lc) - \prod_{l=0}^{n-1} (\alpha + lc).
\]

\begin{figure}[!htb]
\centering
\def\svgwidth{0.525\textwidth}
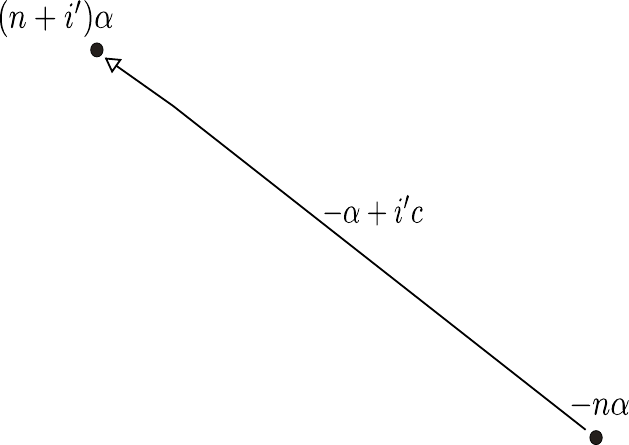
\caption{Another edge of the stable moment graph of $\widehat{\mathfrak{sl}_{2}}$} \label{fig:stable_edge_}
\end{figure}

It suffices to show that
\begin{equation} \label{eq:zeroo}
(a_{i', n}i'c+i'b_{i', n}c) \cdot \prod_{l=i'+1}^{i'+n-1} (-i'c + lc) - \prod_{l=0}^{n-1} (i'c + lc) = 0.
\end{equation}
We have
\begin{equation*}
\begin{split}
& (a_{i', n}i'c+i'b_{i', n}c) \cdot \prod_{l=i'+1}^{i'+n-1} (-i'c + lc) = i'(a_{i', n}+b_{i', n})(n-1)!c^{n}=\\
& = i'\bigg(\binom{n+i'-1}{n}(n-1)+\binom{n+i'-1}{n}\frac{i'-(i'-1)n}{i'}\bigg)(n-1)!c^{n}=\\
& = i'\binom{n+i'-1}{n}\bigg(n-1+\frac{i'-(i'-1)n}{i'}\bigg)(n-1)!c^{n}=\\
& = i'\binom{n+i'-1}{n}\frac{i'(n-1)+i'-(i'-1)n}{i'}(n-1)!c^{n}=\\
& = i'\binom{n+i'-1}{n}\frac{n}{i'}(n-1)!c^{n}=\binom{n+i'-1}{n}n!c^{n}=\\
& = i'(i'+1)(i'+2)\cdots(i'+n-1)c^{n} = \prod_{l=0}^{n-1} (i'c + lc).
\end{split} 
\end{equation*}
This proves~\eqref{eq:zeroo}.

Next, let $i, i' \in \mathbb{Z}_{>0}$. If $i=i'$, then
\[
\alpha \bigm\vert (a_{i, n}\alpha + ib_{i, n}c) \cdot \prod_{l=i+1}^{i+n-1} (-\alpha + lc) - b_{i, n} \cdot \prod_{l=i}^{i+n-1} (\alpha + lc),
\]
i.e. the corresponding structure algebra condition is satisfied. If $i' > i$, we need to prove that 
\[
-\alpha+(i'-i)c \bigm\vert (a_{i', n}\alpha+i'b_{i', n}c) \cdot \prod_{l=i'+1}^{i'+n-1} (-\alpha + lc) - b_{i, n} \cdot \prod_{l=i}^{i+n-1} (\alpha + lc).
\]

Again it is enough to show that
\begin{equation} \label{eq:zerooo}
(a_{i', n}(i'-i)c+i'b_{i', n}c) \cdot \prod_{l=i'+1}^{i'+n-1} ((i-i')c + lc) - b_{i, n} \cdot \prod_{l=i}^{i+n-1} ((i'-i)c + lc) = 0.
\end{equation}
We have
\begin{equation*}
\begin{split}
& (a_{i', n}(i'-i)c+i'b_{i', n}c) \cdot \prod_{l=i'+1}^{i'+n-1} ((i-i')c + lc) = \\
& = \bigg(\binom{n+i'-1}{n}(n-1)(i'-i)+i'\binom{n+i'-1}{n}\frac{i'-(i'-1)n}{i'}\bigg) \cdot \prod_{l=1}^{n-1}(i+l) \cdot \\
& \cdot c^{n} = \binom{n+i'-1}{n}\bigg((n-1)(i'-i)+i'-(i'-1)n\bigg) \cdot \prod_{l=1}^{n-1}(i+l) \cdot c^{n} = \\
& = \binom{n+i'-1}{n}\big(-ni+i+n\big) \cdot \prod_{l=1}^{n-1}(i+l) \cdot c^{n} = \\
& = \binom{n+i'-1}{n}\big(i-(i-1)n\big) \cdot \prod_{l=1}^{n-1}(i+l) \cdot c^{n}
\end{split}
\end{equation*}
and
\begin{equation*}
\begin{split}
& b_{i, n} \cdot \prod_{l=i}^{i+n-1} ((i'-i)c + lc) = \\
& = \binom{n+i-1}{n}\frac{i-(i-1)n}{i}i'(i'+1)\cdots(i'+n-1)c^{n}=\\
& = (n+i-1)(n+i-2)\cdots i \frac{i-(i-1)n}{i} \frac{i'(i'+1)\cdots(i'+n-1)}{n!}c^{n}=\\
& = \bigg(\prod_{l=1}^{n-1}(i+l)\bigg) \cdot \big(i-(i-1)n\big)\binom{n+i'-1}{n}c^{n},
\end{split}
\end{equation*}
so we have proved~\eqref{eq:zerooo}.

Last, if $i' < i$, as before, the only thing that changes is a global sign, therefore the divisibility relations to be satisfied are the same, and thus previously proved. 
Regarding the differences $(v_{n})_{j\alpha}-(v_{n})_{j'\alpha}$ for $\abs{j} \leq n-1$, $j=n$, $\abs{j'} \leq n-1$, or $j'=n$, notice that the corresponding divisibility 
relations are obviously satisfied, and with this the proof has been completed.
\end{proof}

\begin{prop} \label{bas}
The elements $u_{n},\, n \in \mathbb{Z}_{\geq 0}$ and $v_{n},\, n \in \mathbb{Z}_{>0}$ of the structure algebra $\mathcal{Z}$ from the Proposition~\textup{\ref{unvn}} 
form a basis of the $S$-module $\mathcal{Z}$.
\end{prop}

\begin{proof} Let us consider an arbitrary element $w$ of the structure algebra $\mathcal{Z}$ of the following form:
\[
w = \begin{pmatrix}
  \vdots & \vdots \\
  f & g \\
  0 & 0 \\
  \vdots & \vdots \\
  0 & 0 \\
  0 & 0
 \end{pmatrix}
 \]
where the first $n$ (for some $n \in \mathbb{Z}_{\geq 0}$) rows are zero and $f$ and $g$ are some elements of the algebra $S$. Since $\displaystyle\prod_{j\alpha \in 
\mathcal{V}(\mathcal{\widehat{G}}^{\text{stab}}_{>0\alpha})}\mkern-18mu S$ \; is a direct product of \emph{graded} algebras and the structure algebra conditions imply that
\begin{equation} \label{eq:sac}
 \prod_{l=1}^{n} (-\alpha + lc) \mid f, \quad \prod_{l=1}^{n} (\alpha + lc) \mid g, \quad \textrm{and} \quad \alpha \mid g-f,
\end{equation}
we may assume without loss of generality that $f$ and $g$ (and all other components) are homogeneous polynomials of degree $m$, for some non-negative integer $m \geq n$. 

If $m=n$, then 
\[
f = s \cdot \prod_{l=1}^{n} (-\alpha + lc), \quad g = s \cdot \prod_{l=1}^{n} (\alpha + lc),
\]
for some $s \in \mathbb{C}, s \neq 0$ (we assume that both $f$ and $g$ are not equal to zero), because of~\eqref{eq:sac}. After subtraction we obtain
\[
w - su_{n} = \begin{pmatrix}
  \vdots & \vdots \\
  \tilde{f} & \tilde{g} \\
  0 & 0 \\
  \vdots & \vdots \\
  0 & 0 \\
  0 & 0
 \end{pmatrix}
 \]
where the first $n+1$ rows are zero and $\tilde{f}$ and $\tilde{g}$ are some homogeneous polynomials of degree $m=n$. Since $\displaystyle\prod_{l=1}^{n+1} 
(-\alpha + lc) \mid \tilde{f}$ and $\displaystyle\prod_{l=1}^{n+1} (\alpha + lc) \mid \tilde{g}$, the polynomials $\tilde{f}$ and $\tilde{g}$ must be equal to zero. 
The same is true for all other components, i.e. $w = su_{n}$. (If $f$ were equal to zero, we would consider $v_{n+1}$ instead of $u_{n}$ and similarly, if $g$ 
were equal to zero, we would consider $\alpha u_{n} - v_{n+1}$ instead of $u_{n}$.)

Now if $m > n$, because we again have~\eqref{eq:sac}, it follows that
\[
\frac{g \cdot \displaystyle\prod_{l=1}^{n} (-\alpha + lc) - f \cdot \displaystyle\prod_{l=1}^{n} (\alpha + lc)}{\displaystyle\prod_{l=1}^{n} (\alpha + lc) \cdot 
\displaystyle\prod_{l=1}^{n} (-\alpha + lc)} = \frac{g}{\displaystyle\prod_{l=1}^{n} (\alpha + lc)} - 
\frac{f}{\displaystyle\prod_{l=1}^{n} (-\alpha + lc)}
\]
is a polynomial in $S$, i.e.
\[
\prod_{l=1}^{n} (\alpha + lc) \cdot \prod_{l=1}^{n} (-\alpha + lc) \bigm\vert g \cdot \prod_{l=1}^{n} (-\alpha + lc) - 
f \cdot \prod_{l=1}^{n} (\alpha + lc)
\]
and 
\[
\alpha \bigm\vert g \cdot \prod_{l=1}^{n} (-\alpha + lc) - f \cdot \prod_{l=1}^{n} (\alpha + lc).
\]
Therefore we can define the following homogeneous polynomials:
\[x \coloneqq \frac{f}{\displaystyle\prod_{l=1}^{n} (-\alpha + lc)}\]
and
\[
y \coloneqq \frac{g \cdot \displaystyle\prod_{l=1}^{n} (-\alpha + lc) - f \cdot \displaystyle\prod_{l=1}^{n} (\alpha + lc)}{\alpha \cdot 
\displaystyle\prod_{l=1}^{n} (\alpha + lc) \cdot \displaystyle\prod_{l=1}^{n} (-\alpha + lc)} = \frac{1}{\alpha} \Bigg(\frac{g}{\displaystyle\prod_{l=1}^{n} 
(\alpha + lc)} - \frac{f}{\displaystyle\prod_{l=1}^{n} (-\alpha + lc)}\Bigg).
\]
Now we have
\begin{equation*}
\begin{split}
xu_{n} + yv_{n+1} &= \\
 = x \begin{pmatrix}
  \vdots & \vdots \\
  \displaystyle\prod_{l=1}^{n} (-\alpha + lc) & \displaystyle\prod_{l=1}^{n} (\alpha + lc) \\
  0 & 0 \\
  \vdots & \vdots \\
  0 & 0 \\
  0 & 0
 \end{pmatrix} &+ y \begin{pmatrix}
  \vdots & \vdots \\
  0 & \displaystyle\prod_{l=0}^{n} (\alpha + lc) \\
  0 & 0 \\
  \vdots & \vdots \\
  0 & 0 \\
  0 & 0
 \end{pmatrix} 
= \begin{pmatrix}
  \vdots & \vdots \\
  f & g \\
  0 & 0 \\
  \vdots & \vdots \\
  0 & 0 \\
  0 & 0
 \end{pmatrix}.
\end{split}
\end{equation*}
If the obtained element of the structure algebra $\mathcal{Z}$ differs from $w$, then we apply the same procedure to $w - (xu_{n} + yv_{n+1})$. We stop this procedure 
when we obtain the element that has $m$ zero rows, i.e. the degree of the components and the number of zero rows coincide and that brings us to the above case.

Hence we have shown that the set $\{ u_{n} \mid n \in \mathbb{Z}_{\geq 0} \} \cup \{ v_{n} \mid n \in \mathbb{Z}_{>0} \}$ generates the $S$-module $\mathcal{Z}$. Since 
the linear independence is obvious, we conclude that the set $\{ u_{n} \mid n \in \mathbb{Z}_{\geq 0} \} \cup \{ v_{n} \mid n \in \mathbb{Z}_{>0} \}$ is indeed a basis 
of the $S$-module $\mathcal{Z}$.
\end{proof}

\subsection{``Setting \emph{c} equal to zero''.} \label{subsec:ctozero} We denote by $\mathcal{G}^{\text{fin}}$ the moment graph associated with the finite root system 
$A_{1}$, i.e. the root system $\{ \alpha, -\alpha \}$ (cf.~\cite{Fie16}). This is the graph with two vertices that are connected by one edge labeled with $\alpha$. Then 
we denote by $S^{\text{fin}}=\mathbb{C}[\alpha]$ the corresponding symmetric algebra and by $\mathcal{Z}^{\text{fin}}$ the structure algebra of $\mathcal{G}^{\text{fin}}$, 
i.e.
\[
\mathcal{Z}^{\text{fin}} = \bigg\{ (z_{1}, z_{2}) \biggm\vert z_{1}, z_{2} \in S^{\text{fin}}, \ \alpha \mid z_{1}-z_{2} \bigg\}.
\] 

We define now the $S^{\text{fin}}$-module $\mathcal{Z}_{c=0}$ by ``setting $c$ equal to zero'' in the $S$-module $\mathcal{Z}$ in the following way. 

For $f \in S=\mathbb{C}[\alpha, c]$ we define $\overline{f} \coloneqq f\bigr\vert_{c=0} \in S^{\text{fin}}=\mathbb{C}[\alpha]$. Then 
\[
\mathcal{Z}_{c=0} \coloneqq \bigg\{ (\overline{z_{k\alpha}})_{k \in \mathbb{Z}\setminus\{0\}}  \biggm\vert (z_{k\alpha})_{k \in \mathbb{Z}\setminus\{0\}} \in 
\mathcal{Z} \bigg\}.
\]
The elements $\overline{u_{n}} \coloneqq (\overline{(u_{n})_{k\alpha}})_{k \in \mathbb{Z}\setminus\{0\}}, \, n \in \mathbb{Z}_{\geq 0}$ and 
$\overline{v_{n}} \coloneqq (\overline{(v_{n})_{k\alpha}})_{k \in \mathbb{Z}\setminus\{0\}}, \, n \in \mathbb{Z}_{>0}$ form a basis of the $S^{\text{fin}}$-module 
$\mathcal{Z}_{c=0}$.

Indeed, because $u_{n},\, n \in \mathbb{Z}_{\geq 0}$ and $v_{n},\, n \in \mathbb{Z}_{>0}$ are generators of the $S$-module $\mathcal{Z}$, 
$\overline{u_{n}},\, n \in \mathbb{Z}_{\geq 0}$ and $\overline{v_{n}},\, n \in \mathbb{Z}_{>0}$ are generators of the $S^{\text{fin}}$-module $\mathcal{Z}_{c=0}$. 
Furthermore, observe that ``the row echelon form'' of $u_{n},\, n \in \mathbb{Z}_{\geq 0}$ and $v_{n},\, n \in \mathbb{Z}_{>0}$ is preserved under ``setting $c$ equal 
to zero'', which means that $\overline{u_{n}},\, n \in \mathbb{Z}_{\geq 0}$ and $\overline{v_{n}},\, n \in \mathbb{Z}_{>0}$ are of the same form. For this reason, 
$\overline{u_{n}},\, n \in \mathbb{Z}_{\geq 0}$ and $\overline{v_{n}},\, n \in \mathbb{Z}_{>0}$ are $S^{\text{fin}}$-linearly independent in $\mathcal{Z}_{c=0}$. 

Let $\overline{w}=(\overline{w_{k\alpha}})_{k \in \mathbb{Z}\setminus\{0\}}$ be an arbitrary element of $\mathcal{Z}_{c=0}$, i.e.
\[
\overline{w} = \begin{pmatrix}
  \vdots & \vdots \\
  \overline{w_{k\alpha}} & \overline{w_{-k\alpha}} \\
  \vdots & \vdots \\
  \overline{w_{2\alpha}} & \overline{w_{-2\alpha}} \\
  \overline{w_{\alpha}} & \overline{w_{-\alpha}}
 \end{pmatrix}.
\]
We regard the rows of $\overline{w}$ as ordered pairs:
\[
\mathbf{a}_{j} \coloneqq (\overline{w_{(j+1)\alpha}}, \overline{w_{(-j-1)\alpha}}), \ j \in \mathbb{Z}_{\geq 0}.
\]
Observe that $\mathbf{a}_{j} \in \mathcal{Z}^{\text{fin}}$ for all non-negative integers $j$ as $\overline{w_{(j+1)\alpha}}-\overline{w_{(-j-1)\alpha}}$ is divisible 
by $\alpha$. Now we have the following theorem.

\begin{thm} \label{main}
\begin{equation}  \label{eq:divisibility}
\mathcal{Z}_{c=0} = \Bigg\{ (\mathbf{a}_{j}) \in \prod_{j \in \mathbb{Z}_{\geq 0}} \mathcal{Z}^{\textnormal{fin}}
                         \biggm\vert \ \sum_{j=0}^{m} (-1)^{j} \binom{m}{j} \mathbf{a}_{j} \in ((-\alpha)^{m}, \alpha^{m}) 
                         \mathcal{Z}^{\textnormal{fin}} \ \forall m \in \mathbb{Z}_{\geq 0} \Bigg\}.
\end{equation}                        
\end{thm}

\begin{proof}
In order to prove that $\mathcal{Z}_{c=0}$ is a subset of the set given on the right-hand side, it is enough to show that the basis elements 
$\overline{u_{n}}, \, n \in \mathbb{Z}_{\geq 0}$ and $\overline{v_{n}}, \, n \in \mathbb{Z}_{>0}$ of the $S^{\text{fin}}$-module $\mathcal{Z}_{c=0}$ 
satisfy divisibility relations.

Let 
\[
\overline{u_{n}} = \begin{pmatrix}
  \vdots & \vdots \\
  k_{3, n}(-\alpha)^{n} & k_{3, n}\alpha^{n} \\
  k_{2, n}(-\alpha)^{n} & k_{2, n}\alpha^{n} \\
  k_{1, n}(-\alpha)^{n} & k_{1, n}\alpha^{n} \\
  (-\alpha)^{n} & \alpha^{n} \\
  0 & 0 \\
  \vdots & \vdots \\
  0 & 0 \\
  0 & 0
 \end{pmatrix}
 \]
be a basis element of the $S^{\text{fin}}$-module $\mathcal{Z}_{c=0}$, where the first $n$ rows are zero for an arbitrary $n \in \mathbb{Z}_{\geq 0}$ and 
$k_{i, n}, i \in \mathbb{Z}_{>0}$ are as in~\eqref{eq:k_i}.

If $m \leq n$, then the corresponding divisibility relation is obviously satisfied. 

Therefore let us assume that $m=n+k$ for some $k \in \mathbb{Z}_{>0}$.
We want to show that $((-\alpha)^{n+k}, \alpha^{n+k})$ divides \[ \sum_{j=0}^{n+k} (-1)^{j} \binom{n+k}{j} \mathbf{a}_{j}. \]
The sum above equals to 
\[
\sum_{j=n}^{n+k} (-1)^{j} \binom{n+k}{j} \binom{j}{n} ((-\alpha)^{n}, \alpha^{n}), 
\]
so we need to prove that
\begin{equation} \label{eq:bnmzero}
\sum_{j=n}^{n+k} (-1)^{j} \binom{n+k}{j} \binom{j}{n} = 0.
\end{equation}
We have
\begin{equation*}
\begin{split}
& \sum_{l=0}^{k} (-1)^{n+l} \binom{n+k}{n+l} \binom{n+l}{n} = \sum_{l=0}^{k} (-1)^{n+l} \binom{n+k}{n} \binom{k}{l} = \\
& = (-1)^{n} \binom{n+k}{n} \underbrace{\sum_{l=0}^{k} (-1)^{l} \binom{k}{l}}_{=0} = 0, 
\end{split}
\end{equation*}
which proves~\eqref{eq:bnmzero}.

Now let 
\[
\overline{v_{n}} = \begin{pmatrix}
  \vdots & \vdots \\
  -a_{3, n}(-\alpha)^{n} & b_{3, n}\alpha^{n} \\
  -a_{2, n}(-\alpha)^{n} & b_{2, n}\alpha^{n} \\
  -a_{1, n}(-\alpha)^{n} & b_{1, n}\alpha^{n} \\
  0 & \alpha^{n} \\
  0 & 0 \\
  \vdots & \vdots \\
  0 & 0 \\
  0 & 0
 \end{pmatrix}
 \]
be a basis element of the $S^{\text{fin}}$-module $\mathcal{Z}_{c=0}$, where the first $n-1$ rows are zero for an arbitrary $n \in \mathbb{Z}_{>0}$ and 
$a_{i, n}, b_{i, n}, i \in \mathbb{Z}_{>0}$ are as in~\eqref{eq:ab_i}.

If $m \leq n-2$, then the corresponding divisibility relation is trivially satisfied.

If $m = n-1$, then we have
\[
\sum_{j=0}^{n-1} (-1)^{j} \binom{n-1}{j} \mathbf{a}_{j} = (-1)^{n-1} (0, \alpha^{n}) = \underbrace{(0, (-1)^{n-1}\alpha)}_{\in \mathcal{Z}^{\text{fin}}}
((-\alpha)^{n-1}, \alpha^{n-1}). 
\]

If $m = n$, we have
\begin{equation*}
\begin{split}
& \sum_{j=0}^{n} (-1)^{j} \binom{n}{j} \mathbf{a}_{j} = (-1)^{n-1} \binom{n}{n-1} (0, \alpha^{n}) + (-1)^{n} (-a_{1, n}(-\alpha)^{n}, b_{1, n}\alpha^{n}) = \\
& = \Big(0, (-1)^{n-1}n \alpha^{n}\Big) + \Big((-1)^{n-1}(n-1) (-\alpha)^{n}, (-1)^{n} \alpha^{n}\Big) = \\
& = \underbrace{\Big((-1)^{n-1}(n-1), (-1)^{n-1}n + (-1)^{n}\Big)}_{\in \mathcal{Z}^{\text{fin}}}\Big((-\alpha)^{n}, \alpha^{n}\Big).
\end{split}
\end{equation*}

Now let $m=n+k$ for some $k \in \mathbb{Z}_{>0}$. Then
\begin{equation*}
\begin{split}
\sum_{j=0}^{n+k} (-1)^{j} \binom{n+k}{j} \mathbf{a}_{j} &= \\
= (-1)^{n-1} \binom{n+k}{n-1} (0, \alpha^{n}) &+ \sum_{j=n}^{n+k} (-1)^{j} \binom{n+k}{j} (-a_{j-n+1, n}(-\alpha)^{n}, b_{j-n+1, n}\alpha^{n}),
\end{split}
\end{equation*}
hence we need to show that 
\[ (-1)^{n-1} \binom{n+k}{n-1} \!(0, \alpha^{n}) + \sum_{j=n}^{n+k} (-1)^{j} \binom{n+k}{j} \!(-a_{j-n+1, n}(-\alpha)^{n}, b_{j-n+1, n}\alpha^{n}) = 0, \]
i.e. 
\begin{equation} \label{eq:bnmzeroo}
\sum_{j=n}^{n+k} (-1)^{j} \binom{n+k}{j} a_{j-n+1, n} = 0
\end{equation}
and 
\begin{equation} \label{eq:bnmzerooo}
(-1)^{n-1}\binom{n+k}{n-1} +  \sum_{j=n}^{n+k} (-1)^{j} \binom{n+k}{j} b_{j-n+1, n} = 0.
\end{equation} 

The left-hand side of~\eqref{eq:bnmzeroo} equals
\begin{equation*}
\begin{split}
& \sum_{j=n}^{n+k} (-1)^{j} \binom{n+k}{j} \binom{j}{n} (n-1) = \\
& = (n-1) \sum_{l=0}^{k} (-1)^{n+l} \binom{n+k}{n+l} \binom{n+l}{n} = \\
& = (-1)^{n}(n-1)\binom{n+k}{n} \underbrace{\sum_{l=0}^{k} (-1)^{l} \binom{k}{l}}_{=0} = 0. 
\end{split} 
\end{equation*}
Therefore~\eqref{eq:bnmzeroo} is proved.

Then on the left-hand side of~\eqref{eq:bnmzerooo} we have the following:
\begin{equation} \label{eq:sum_bi}
\begin{split}
& (-1)^{n-1}\binom{n+k}{n-1} + \sum_{j=n}^{n+k} (-1)^{j} \binom{n+k}{j} \binom{j}{n} \frac{(j-n+1)-(j-n)n}{j-n+1} = \\
& = (-1)^{n-1}\binom{n+k}{n-1} + \sum_{j=n}^{n+k} (-1)^{j} \binom{n+k}{j} \binom{j}{n} \bigg( 1 - \frac{j-n}{j-n+1} n \bigg) = \\
& = (-1)^{n-1}\binom{n+k}{n-1} + \sum_{l_{1}=0}^{k} (-1)^{n+l_{1}} \binom{n+k}{n+l_{1}} \binom{n+l_{1}}{n} - \\
& - \sum_{l_{2}=0}^{k} (-1)^{n+l_{2}} \binom{n+k}{n+l_{2}} \binom{n+l_{2}}{n} \frac{l_{2}}{l_{2}+1} n = \\
& = (-1)^{n-1}\binom{n+k}{n-1} + (-1)^{n}\binom{n+k}{n} \underbrace{\sum_{l_{1}=0}^{k} (-1)^{l_{1}}\binom{k}{l_{1}}}_{=0} - \\ 
& - (-1)^{n}n\binom{n+k}{n} \underbrace{\sum_{l_{2}=0}^{k} (-1)^{l_{2}}\binom{k}{l_{2}}\frac{l_{2}}{l_{2}+1}}_\text{(\textbullet)}.
\end{split}
\end{equation}

Now we calculate the sum (\textbullet).
\begin{equation*}
\begin{split}
& \sum_{l_{2}=0}^{k} (-1)^{l_{2}}\binom{k}{l_{2}}\frac{l_{2}}{l_{2}+1} = \sum_{l_{2}=0}^{k} (-1)^{l_{2}}\binom{k+1}{l_{2}+1}\frac{l_{2}}{k+1} =\\
& = -\frac{1}{k+1} \sum_{l_{2}=0}^{k} (-1)^{l_{2}+1}\binom{k+1}{l_{2}+1}l_{2} = -\frac{1}{k+1} \sum_{l_{3}=1}^{k+1} (-1)^{l_{3}}\binom{k+1}{l_{3}}(l_{3}-1) =\\
& = -\frac{1}{k+1} \Bigg( \underbrace{\sum_{l_{3}=1}^{k+1} (-1)^{l_{3}}\binom{k+1}{l_{3}}l_{3}}_\text{(\textbullet \textbullet)} - 
\underbrace{\sum_{l_{3}=1}^{k+1} (-1)^{l_{3}}\binom{k+1}{l_{3}}}_\text{(\textbullet \textbullet \textbullet)} \Bigg).
\end{split}
\end{equation*}
The sum (\textbullet \textbullet) is equal to
\begin{equation*}
\begin{split}
& \sum_{l_{3}=1}^{k+1} (-1)^{l_{3}}(k+1)\binom{k}{l_{3}-1} = (k+1) \sum_{l_{4}=0}^{k} (-1)^{l_{4}+1} \binom{k}{l_{4}} = \\
& = -(k+1) \underbrace{\sum_{l_{4}=0}^{k} (-1)^{l_{4}} \binom{k}{l_{4}}}_{=0} = 0
\end{split}
\end{equation*}
and the sum (\textbullet \textbullet \textbullet) is equal to
\[
\underbrace{\sum_{l_{3}=0}^{k+1} (-1)^{l_{3}}\binom{k+1}{l_{3}}}_{=0} - 1 = -1.
\]
Therefore the sum (\textbullet) equals $-\dfrac{1}{k+1}$ and we have in~\eqref{eq:sum_bi} the following:
\begin{equation*}
\begin{split}
& (-1)^{n-1}\binom{n+k}{n-1} + (-1)^{n}n\binom{n+k}{n}\frac{1}{k+1} = \\
& = (-1)^{n-1}\binom{n+k}{n-1} + (-1)^{n}\binom{n+k}{n-1} = \\
& = \binom{n+k}{n-1} \underbrace{\Bigg((-1)^{n-1}+(-1)^{n}\Bigg)}_{=0} = 0. 
\end{split}
\end{equation*}
Thus we have proved~\eqref{eq:bnmzerooo}.

Finally, to prove that the $S^{\text{fin}}$-module given on the right-hand side of~\eqref{eq:divisibility} is included in $\mathcal{Z}_{c=0}$, we will show 
that the basis elements $\overline{u_{n}}, \, n \in \mathbb{Z}_{\geq 0}$ and $\overline{v_{n}}, \, n \in \mathbb{Z}_{>0}$ of $\mathcal{Z}_{c=0}$ generate that 
$S^{\text{fin}}$-module (by performing a procedure similar to Gaussian elimination).

For $n \geq 1$, because of the grading on the direct product $\mkern-15mu\displaystyle\prod_{j\alpha \in \mathcal{V}(\mathcal{\widehat{G}}^{\text{stab}}_{>0\alpha})}
\mkern-18mu S$\, (and on the direct product $\mkern-15mu\displaystyle\prod_{j\alpha \in \mathcal{V}(\mathcal{\widehat{G}}^{\text{stab}}_{>0\alpha})}\mkern-18mu 
S^{\text{fin}}$\,), it is enough to consider homogeneous elements, hence of the following form:
\[
 \begin{pmatrix}
  \vdots & \vdots \\
  x_{2}\alpha^{n} & y_{2}\alpha^{n} \\
  x_{1}\alpha^{n} & y_{1}\alpha^{n} \\
  x_{0}\alpha^{n} & y_{0}\alpha^{n}
 \end{pmatrix}
 \]
where $x_{i}, y_{i} \in \mathbb{C}, i \in \mathbb{Z}_{\geq 0}$.
We have the following:
\[
\begin{pmatrix}
  \vdots & \vdots \\
  x_{2}\alpha^{n} & y_{2}\alpha^{n} \\
  x_{1}\alpha^{n} & y_{1}\alpha^{n} \\
  x_{0}\alpha^{n} & y_{0}\alpha^{n}
\end{pmatrix}
 - x_{0}\alpha^{n}
\underbrace{\begin{pmatrix}
  \vdots & \vdots \\
  1 & 1 \\
  1 & 1 \\
  1 & 1
 \end{pmatrix}}_{=\overline{u_{0}}}
 =
\begin{pmatrix}
  \vdots & \vdots \\
  (x_{2}-x_{0})\alpha^{n} & (y_{2}-x_{0})\alpha^{n} \\
  (x_{1}-x_{0})\alpha^{n} & (y_{1}-x_{0})\alpha^{n} \\
  0 & (y_{0}-x_{0})\alpha^{n}
 \end{pmatrix}
\] 
\begin{equation*}
\begin{split}
\begin{pmatrix}
  \vdots & \vdots \\
  (x_{2}-x_{0})\alpha^{n} & (y_{2}-x_{0})\alpha^{n} \\
  (x_{1}-x_{0})\alpha^{n} & (y_{1}-x_{0})\alpha^{n} \\
  0 & (y_{0}-x_{0})\alpha^{n}
 \end{pmatrix}
 - (y_{0}-x_{0})\alpha^{n-1}
\underbrace{\begin{pmatrix}
  \vdots & \vdots \\
  0 & \alpha \\
  0 & \alpha \\
  0 & \alpha
 \end{pmatrix}}_{=\overline{v_{1}}}
 = \\
= \begin{pmatrix}
  \vdots & \vdots \\
  (x_{2}-x_{0})\alpha^{n} & (y_{2}-y_{0})\alpha^{n} \\
  (x_{1}-x_{0})\alpha^{n} & (y_{1}-y_{0})\alpha^{n} \\
  0 & 0
 \end{pmatrix}
\end{split}
\end{equation*}

\[ \vdots \]
The procedure stops at:
\[
\begin{pmatrix}
  \vdots & \vdots \\
  z_{2}\alpha^{n} & z_{3}\alpha^{n} \\
  0 & z_{1}\alpha^{n} \\
  0 & 0 \\
  \vdots & \vdots \\
  0 & 0
\end{pmatrix}
\]
where the first $n$ rows are zero and $z_{i} \in \mathbb{C}, \ i \in \mathbb{Z}_{>0}$. Since $\displaystyle \sum_{j=0}^{n} (-1)^{j} \binom{n}{j} \mathbf{a}_{j} 
\in ((-\alpha)^{n}, \alpha^{n})\mathcal{Z}^{\text{fin}}$, it follows that
\[ (-1)^{n}(0, z_{1}\alpha^{n}) = \underbrace{(u,u)}_{\in \mathcal{Z}^{\text{fin}}}((-\alpha)^{n}, \alpha^{n}) \ \text{ for some } u \in \mathbb{C}. \] 
Therefore $z_{1}$ is equal to zero. By applying the divisibility relations for $m \geq n+1$, we conclude that all rows are zero.

For $n=0$, let 
\[
 \begin{pmatrix}
  \vdots & \vdots \\
  c_{2} & c_{2} \\
  c_{1} & c_{1} \\
  c_{0} & c_{0}
 \end{pmatrix},
 \quad
 c_{i} \in \mathbb{C}, \ i \in \mathbb{Z}_{\geq 0}
 \]
be an arbitrary element of the $S^{\text{fin}}$-module on the right-hand side of~\eqref{eq:divisibility}. Then
\[ 
\begin{pmatrix}
  \vdots & \vdots \\
  c_{2} & c_{2} \\
  c_{1} & c_{1} \\
  c_{0} & c_{0}
 \end{pmatrix}
 - c_{0}
\underbrace{\begin{pmatrix}
  \vdots & \vdots \\
  1 & 1 \\
  1 & 1 \\
  1 & 1
 \end{pmatrix}}_{=\overline{u_{0}}}
 =
\begin{pmatrix}
  \vdots & \vdots \\
  c_{2}-c_{0} & c_{2}-c_{0} \\
  c_{1}-c_{0} & c_{1}-c_{0} \\
  0 & 0
 \end{pmatrix} 
\]
and again by applying the divisibility relations for $m \geq 1$, we obtain \[ c_{i}=c_{0} \quad \forall i \in \mathbb{Z}_{>0}. \] 
This completes the proof.
\end{proof}

\begin{rem}
The structure algebra $\mathcal{Z}$ is an $S$-\emph{algebra} and the $S^{\text{fin}}$-module $\mathcal{Z}_{c=0}$ inherits the algebra structure. Therefore 
$\mathcal{Z}_{c=0}$ is an $S^{\text{fin}}$-\emph{algebra}, where the multiplication is given componentwise. This algebra structure on $\mathcal{Z}_{c=0}$ 
implies that higher-order congruence relations from Theorem~\ref{main} are closed under multiplication. This property is far from obvious. 
\end{rem}

\section{ACKNOWLEDGEMENTS} 

I would like to thank Peter Fiebig for helpful discussions and interesting ideas about sheaves on moment graphs. This article was partially written when the author 
held a postdoctoral position funded by the research project BIRD179758/17 of the University of Padova. This research was also partially funded by DOR1717189/17 of the 
same university. In addition, this research was supported by the DFG Priority Programme 1388 and the project PRG34672 funded by the MIUR-DAAD Programme 2018/2019.

\end{document}